\documentclass{amsart}

\usepackage{amssymb}
\usepackage{amsthm}
\usepackage{amsmath}
\usepackage{amsbsy}
\usepackage[all]{xy}
\usepackage{bm}
\usepackage{hyperref}
\usepackage{tikz}
\usepackage{array}
\usepackage{colortbl,xcolor}
\usepackage{ytableau}
\usepackage{hhline}
\allowdisplaybreaks

\newtheorem{theorem}{Theorem}[section]

\newtheorem{lemma}[theorem]{Lemma}
\newtheorem{question}[theorem]{Question}

\DeclareMathOperator{\stab}{Stab}

\begin{document}

\title[]{The largest $(k, \ell)$-sum-free sets\\ in compact abelian groups} 

\author[]{Noah Kravitz}
\address[]{Grace Hopper College, Yale University, New Haven, CT 06510, USA}
\email{noah.kravitz@yale.edu}

\begin{abstract} 
A subset $A$ of a finite abelian group is called $(k,\ell)$-sum-free if $kA \cap \ell A=\emptyset.$  In this paper, we extend this concept to compact abelian groups and study the question of how large a measurable $(k,\ell)$-sum-free set can be.  For integers $1 \leq k <\ell$ and a compact abelian group $G$, let $$\lambda_{k,\ell}(G)=\sup\{ \mu(A): kA \cap \ell A =\emptyset \}$$ be the maximum possible size of a $(k,\ell)$-sum-free subset of $G$.  We prove that if $G=\mathbb{I} \times M$, where $\mathbb{I}$ is the identity component of $G$, then $$\lambda_{k, \ell}(G)=\max \left\{ \lambda_{k, \ell}(M), \lambda_{k, \ell}(\mathbb{I}) \right\}.$$  Moreover,  if $\mathbb{I}$ is nontrivial, then $\lambda_{k,\ell}(\mathbb{I})=\frac{1}{k+\ell}$.  Finally, we discuss how this problem motivates a new framework for studying $(k,\ell)$-sum-free sets in finite groups.
\end{abstract}
\maketitle

\section{Introduction}

The \textit{Minkowski sum} of two subsets $A,B$ of an additive abelian group $G$ is $$A+B=\{a+b:a \in A, b \in B\}.$$  When $G$ is finite, a natural question is how large a subset $A \subset G$ can be \textit{sum-free}, i.e., can satisfy $(A+A) \cap A=\emptyset$.  In other words, $A$ is sum-free if $x+y=z$ has no solutions for $x,y,z \in A$.  Early progress on this question for cyclic groups appears in the work of Diananda and Yap \cite{dianandayap} and Wallis, Street, and Wallis \cite{wallisstreetwallis}.  In 2005, Green and Ruzsa \cite{greenruzsa} completely solved this problem for abelian groups.  Let $\lambda_{1,2}(G)$ denote the maximum density of a sum-free subset of $G$.

\begin{theorem}[Green and Ruzsa 2005]\label{thm:greenruzsa}
For any finite abelian group $G$ with exponent $\exp(G)$, we have
$$\lambda_{1,2}(G)=\max_{d|\exp(G)} \left\{ \left\lceil \frac{d-1}{3} \right\rceil \cdot \frac{1}{d} \right\}.$$
In particular, $\frac{2}{7}\leq \lambda_{1,2}(G)\leq \frac{1}{2}$, both bounds of which are sharp.
\end{theorem}

Other statistics on sum-free sets have been the object of considerable study (see, e.g., \cite{green, taovu} and the references therein).
\\

This problem has more recently been generalized to $(k, \ell)$-sum-free sets.  For a positive integer $k$, let $kA=\underbrace{A+\cdots+A}_k$ denote the $k$-fold Minkowski sum of $A$ with itself (NOT the $k$-fold dilation of $A$).  Then, for a finite abelian group $G$, let $$\lambda_{k,\ell}(G)=\max \left\{\frac{|A|}{|G|}: kA \cap \ell A=\emptyset \right\}$$ denote the maximum density of a $(k,\ell)$-sum-free subset of $G$.  Trivially, $\lambda_{k,\ell}(G)=0$ when $k=\ell$, so by convention we take $1\leq k <\ell$.
\\

Most work has focused on $(k,\ell)$-sum-free sets in cyclic groups; the general abelian case remains far from understood.  Important results are due to Bier and Chin \cite{bierchin} and Hamidoune and Plagne \cite{hamidouneplagne}, whose approaches relied on Vosper's Theorem and Kneser's Theorem.  In 2018, Bajnok and Matzke \cite{bajnokmatzke} found a general expression for $\lambda_{k,\ell}(\mathbb{Z}_n)$ by analyzing $(k,\ell)$-sum-free arithmetic progressions.

\begin{theorem}[Bajnok and Matzke 2018]\label{thm:bajnokmatzke}
For any integers $1 \leq k <\ell$ and $n \geq 1$, we have
$$\lambda_{k,\ell}(\mathbb{Z}_n)=\max_{d|n} \left\{ \left\lceil \frac{d-\delta_d+r_d}{k+\ell} \right\rceil \cdot \frac{1}{d} \right\},$$
where $\delta_d =\gcd(d, \ell-k)$ and $r_d$ is the remainder of $k \left\lceil \frac{d-\delta_d}{k+\ell} \right\rceil$ modulo $\delta_d$.
\end{theorem}

For further background, see the excellent exposition in \cite{bajnok}.
\\

One might wonder about the analogous problem on the circle group $\mathbb{T}=\mathbb{R}/\mathbb{Z}$ (with the Euclidean topology) and the $d$-dimensional torus $\mathbb{T}^d$.  In this paper, we generalize the study of $(k,\ell)$-sum-free sets to compact abelian groups.  For a compact abelian group $G$, let $\mu$ be the probability Haar measure (normalized so that $\mu(G)=1$).  We then define
$$\lambda_{k,\ell}(G)=\sup\{ \mu(A): kA \cap \ell A =\emptyset \},$$
where the supremum runs over all measurable subsets $A \subset G$.  Note that when $G$ is finite, this definition coincides with the definition above.  (Previous generalizations of topics in additive combinatorics to a continuous setting include analogs of Mann's Theorem \cite{kemperman} and Freiman's Theorem \cite{roton}.)
\\

The main result of this paper is the following formula for $\lambda_{k, \ell}(G)$ when $G$ can be written as $G=\mathbb{I} \times M$, where $\mathbb{I}$ is the identity component of $G$ (and $M$ is finite or profinite).

\begin{theorem}[Main Theorem]\label{thm:main}
For any integers $1\leq k<\ell$ and any compact abelian group $G=\mathbb{I} \times M$, we have
$$\lambda_{k,\ell}(G)=\max \left\{ \lambda_{k, \ell}(M), \lambda_{k, \ell}(\mathbb{I}) \right\}.$$
Also, if $\mathbb{I}$ is nontrivial, then $\lambda_{k, \ell}(\mathbb{I})=\frac{1}{k+\ell}$ and
$$\lambda_{k,\ell}(G)=\max \left\{ \lambda_{k, \ell}(M), \frac{1}{k+\ell} \right\}.$$
\end{theorem}

Note that positive-dimensional compact abelian Lie groups are included in the latter case.  In particular, $\lambda_{k,\ell}(\mathbb{T}^d)=\frac{1}{k+\ell}$ answers our original question about the $d$-dimensional torus.
\\

In Section $2$, we prove the Main Theorem.  We will make use of the following deep classical result of Kneser \cite{kneser}.  Here, $\mu_{\ast}$ denotes the inner Haar probability measure.  (Even on $\mathbb{T}$, the Minkowski sum of two measurable sets need not be measurable.)

\begin{theorem}[Kneser 1956]\label{thm:kneser}
Let $G$ be a compact abelian group with Haar probability measure $\mu$, and let $A$ and $B$ be nonempty measurable subsets of $G$.  Then
$$\mu_{\ast}(A+B) \geq \min \{ \mu(A)+\mu(B), 1\},$$
unless the stabilizer $H=\stab(A+B)$ is an open subgroup of $G$, in which case
$$\mu_{\ast}(A+B)\geq \mu(A)+\mu(B)-\mu(H).$$
\end{theorem}

In Section $3$, we discuss some consequences of our results and possible future lines of inquiry.  In particular, the compact case inspires a curious new framework for investigating $(k,\ell)$-sum-free sets in the finite context.

\section{proofs}

We begin by recording a few general observations.

\begin{lemma}\label{lem:sumover1}
Let $A$ and $B$ be (not necessarily measurable) subsets of a compact abelian group $G$, with $\mu_{\ast}(A)+\mu_{\ast}(B)>1$.  Then $A+B=G$.
\end{lemma}

\begin{proof}
There exist closed subsets $A_{\ast}\subseteq A$ and $B_{\ast}\subseteq B$ satisfying $\mu(A_{\ast})+\mu(B_{\ast})>1$.  Assume (for the sake of contradiction) that there exists some $g \in G \setminus (A+B)$.  Then $g \notin A_{\ast}+B_{\ast}$, so $A_{\ast}$ and $\{g\}-B_{\ast}$ are disjoint.  But $1 \geq \mu(A_{\ast})+\mu(\{g\}-B_{\ast})=\mu(A_{\ast})+\mu(B_{\ast})$ yields a contradiction.
\end{proof}

\begin{lemma}\label{lem:subgroup}
Let $G=\mathbb{I} \times M$ be a compact abelian group.  Then any open subgroup $H$ of $G$ is of the form $H=\mathbb{I} \times N$, where $N$ is a subgroup of $M$.
\end{lemma}

\begin{proof}
The set $U=\mathbb{I} \cap H$ is open in $\mathbb{I}$ (in the subgroup topology) and nonempty (since it contains the identity of $G$).  Then $U$ and its cosets are an open partition of $\mathbb{I}$.  Since $\mathbb{I}$ is connected, $U=\mathbb{I}$.  Finally, noting that $H/\mathbb{I}$ is a subgroup of $G/\mathbb{I} \cong M$ completes the proof.
\end{proof}

\begin{lemma}\label{lem:stabilizer}
Let $A$ be a nonempty subset of an abelian group $G$.  Then for any integers $1\leq i < j$, we have $\stab(iA) \subseteq \stab(jA)$ as an inclusion of subgroups.
\end{lemma}

\begin{proof}
For any $h \in \stab (iA)$, we have $\{h\}+jA=(\{h\}+iA)+(j-i)A=iA+(j-i)A=jA$, so $h \in \stab(jA)$.
\end{proof}

We now bound $\lambda_{k,\ell}(G)$ from above.

\begin{theorem}\label{thm:upperbound}
For any integers $1 \leq k<\ell$ and any compact abelian group $G=\mathbb{I} \times M$, we have $$\lambda_{k, \ell}(G) \leq \max \left\{ \lambda_{k, \ell}(M), \frac{1}{k+\ell} \right\}.$$
\end{theorem}

\begin{proof}
Assume (for the sake of contradiction) that there exists a $(k, \ell)$-sum-free set $A \subseteq G$ of measure strictly greater than both $\lambda_{k, \ell}(M)$ and $\frac{1}{k+\ell}$.  Consider $H=\stab (\ell A)$, and recall from Lemma \ref{lem:stabilizer} that $\stab(k A) \subseteq H$.  Theorem \ref{thm:kneser} tells us that if $H$ is NOT an open subgroup of $G$, then $\mu_{\ast}(kA) \geq \min \{ k \mu(A), 1\}$ and $\mu_{\ast}(\ell (-A))=\mu_{\ast}(\ell A) \geq \min \{ \ell \mu(A), 1\}$.  Then, since $$\mu_{\ast}(kA)+\mu_{\ast}(\ell A)>\frac{k}{k+\ell}+\frac{\ell}{k+\ell}=1,$$ Lemma \ref{lem:sumover1} gives $kA+\ell(-A)=G$, and, in particular, $0 \in kA +\ell (-A)$.  So there exist $a_1, \ldots, a_{k+\ell} \in A$ satisfying $a_1+\cdots +a_k=a_{k+1}+\cdots +a_{k+\ell}$, which contradicts $A$ being $(k, \ell)$-sum-free.
\\

Next, consider the case where $H$ is an open subgroup of $G$.  By Lemma \ref{lem:subgroup}, we have $H=\mathbb{I} \times N$ for some subgroup $N \subseteq M$.  In particular, $\mathbb{I} \subseteq H$.  Thus, $\ell A$ is a union of cosets of $\mathbb{I}$: if $g \in (m+\mathbb{I})\cap (\ell A)$, then $\mathbb{I}+\{g\}=m+\mathbb{I}\subseteq H+(\ell A)=\ell A$.  Let $$P=\{p \in M: (p+\mathbb{I})\cap A \neq \emptyset\}$$ be the set of the elements of $M$ whose corresponding components of $G$ have nonempty intersection with $A$.  Since $\frac{|P|}{|M|} \geq \mu(A) >\lambda_{k, \ell}(M)$, we have $k P \cap \ell P \neq \emptyset$.  Then there exist $p_1, \ldots, p_{k+\ell} \in P$ and $m \in M$ such that $$p_1+\cdots+p_k=p_{k+1}+\cdots+p_{k+\ell}=m.$$  So there also exist $a_1, \ldots, a_{k+\ell} \in A$ with each $a_i \in p_i+\mathbb{I}$.  Then $a_{k+1}+\cdots +a_{k+\ell} \in (m+\mathbb{I}) \cap (\ell A)$, and we conclude that $m+\mathbb{I} \subseteq \ell A$.  Similarly, $a_1+\cdots +a_k \in (m+\mathbb{I})\cap (k A)$, which contradicts $A$ being $(k, \ell)$-sum-free.  This completes the proof.
\end{proof}

Next, we establish lower bounds through specific constructions of large $(k, \ell)$-sum-free sets.  The following lemma generalizes a common tool in the study of $(k, \ell)$-sum-free sets in finite groups.

\begin{lemma}\label{lem:quotient}
Fix any positive integers $1 \leq k<\ell$ and any compact abelian group $G$.  If $G$ admits a surjective measurable homomorphism $\phi$ onto the topological group $H$, then $\lambda_{k, \ell}(H) \leq \lambda_{k, \ell} (G)$.
\end{lemma}

\begin{proof}
 Let $S \subset H$ be a $(k, \ell)$-sum-free set of density $\mu$.  Then $A=\phi^{-1}(S) \subset G$ is a $(k, \ell)$-sum-free set ($\phi(kA)$ and $\phi(\ell A)$ are disjoint in $H$) with the same density.
\end{proof}

Lower bounds now follow from the obvious choices for $H$.

\begin{lemma}\label{lem:lowerbound}
For any positive integers $1 \leq k<\ell$ and any compact abelian group $G=\mathbb{I} \times M$, we have $$\max \left\{ \lambda_{k, \ell}(M), \lambda_{k, \ell}(\mathbb{I}) \right\} \leq \lambda_{k, \ell}(G).$$
\end{lemma}

\begin{proof}

Apply Lemma \ref{lem:quotient} with $H=M$ and $H=\mathbb{I}$.
\end{proof}

When $\mathbb{I}$ is nontrivial, we can bound $\lambda_{k, \ell}(\mathbb{I})$ from below using Pontryagin duality.

\begin{lemma}\label{lem:character}
For any positive integers $1 \leq k < \ell$ and any nontrivial compact connected abelian group $G$, we have
$$\lambda_{k, \ell}(G) \geq \frac{1}{k+\ell}.$$
\end{lemma}

\begin{proof}
Consider the Pontryagin dual $\hat{G}$ (the group of characters of $G$).  It is well known that a compact abelian group is connected if and only if its Pontryagin dual is torsion-free.  Thus, $\hat{G}$ is torsion-free, and it is nontrivial since $G$ is nontrivial.  Let $\chi \in \hat{G}$ be an element of infinite order.  Since $\chi(G)$ is closed and dense in $\mathbb{T}$, we conclude that $\chi(G)=\mathbb{T}$.  By Lemma \ref{lem:quotient}, we have $\lambda_{k,\ell}(\mathbb{T})\leq \lambda_{k,\ell}(G)$.  Finally, note that the set $$S=\left( \frac{k}{\ell^2-k^2}, \frac{\ell}{\ell^2-k^2} \right)\subset \mathbb{T}$$ is $(k, \ell)$-sum-free (with measure $\frac{1}{k+\ell}$): the intervals $kS=\left( \frac{k^2}{\ell^2-k^2}, \frac{k \ell}{\ell^2-k^2} \right)$ and $\ell S=\left( \frac{k \ell}{\ell^2-k^2}, \frac{\ell^2}{\ell^2-k^2} \right)$ are disjoint in $\mathbb{T}$.
\end{proof}

At last, we show how these results imply the Main Theorem.

\begin{proof}[Proof of the Main Theorem]
We condition on whether or not $\mathbb{I}$ is trivial.  If $\mathbb{I}$ is trivial, then $\lambda_{k,\ell}(\mathbb{I})=0$ and $G$ is isomorphic to $M$.  The first statement of the Main Theorem in this case can now be seen to be tautological.  If $\mathbb{I}$ is nontrivial, then it suffices to observe that the upper bound of Theorem \ref{thm:upperbound} equals the lower bound of Lemma \ref{lem:lowerbound} (using Lemma \ref{lem:character}).
\end{proof}

\section{Discussion}

The Main Theorem bounds $\lambda_{k,\ell}(G)$ in terms of the largest possible $(k,\ell)$-sum-free sets of its ``connected'' and ``discrete'' factors.  In the finite case (cf. Theorem \ref{thm:bajnokmatzke}), one must ordinarily take into consideration the largest $(k,\ell)$-sum-free sets in all subgroups; our Main Theorem shows that for compact $G=\mathbb{I} \times M$, it suffices to look for $(k,\ell)$-sum-free sets in only $\mathbb{I}$ and $M$.
\\

The Main Theorem completely treats $\lambda_{k,\ell}(\mathbb{I})$, but there are still many interesting questions to ask about $\lambda_{k,\ell}(M)$ when $M$ is infinite.  (Recall that the totally disconnected compact groups are the profinite groups.)  Consider, for instance, the case where $M$ is the direct product of countably many finite cyclic groups: $$M=(\mathbb{Z}_2)^{e_2} \times (\mathbb{Z}_3)^{e_3} \times (\mathbb{Z}_4)^{e_4} \times \cdots$$ (where the $e_i$'s are either finite or $\infty$).  Roughly speaking, the measurable subsets of $M$ can be approximated by subsets of the form $S \times (M/H)$, where $H$ is a finite subgroup of $M$ and $S \subseteq H$, so we expect $\lambda_{k,\ell}(M)=\sup \{ \lambda_{k,\ell}(H) \}$ (where $H$ ranges over the finite subgroups of $M$).  As a starting point, Theorem \ref{thm:bajnokmatzke} provides lower bounds.  When $k=1$ and $\ell=2$, we can also apply Theorem \ref{thm:greenruzsa}: for example, $\lambda_{1,2}((\mathbb{Z}_p)^{\infty})=\left\lceil \frac{p-1}{3} \right\rceil \cdot \frac{1}{p}$.
\\

The problem of finding $(k, \ell)$-sum-free subsets when $\mathbb{I}$ is a $d$-dimensional torus and $M$ is finite motivates a set of related questions for finite abelian groups.  Consider the maps $\phi_n: \mathbb{T} \to \mathbb{Z}_n$ via $$A \mapsto \left\{i \in \mathbb{Z}_n: \left( \frac{i}{n}, \frac{i+1}{n} \right) \subseteq A \right\}.$$   The fact $\left( \frac{i}{n}, \frac{i+1}{n} \right)+\left( \frac{j}{n}, \frac{j+1}{n} \right)=\left( \frac{i+j}{n}, \frac{i+j+2}{n} \right)$ implies that for any sets $A, B \subseteq \mathbb{T}$, $$\{0,1\}+\phi_n(A)+\phi_n(B) \subseteq \phi_{n}(A+B),$$ with equality when (but not only when) $A$ and $B$ are the unions of open intervals of the form $\left( \frac{i}{n}, \frac{i+1}{n} \right)$.  (For instance, any open sets can be arbitrarily well approximated from the inside in this manner for large enough $n$.)  This motivates the following set of definitions.
\\

Let $A,B,C$ be subsets of a finite abelian group $G$.  For lack of better notation, let $A \ast_C B=C+A+B$ and $k \ast_C A=\underbrace{A\ast_C \cdots \ast_C A}_k=kA+(k-1)C$.  We can investigate $(k,\ell)$-sum-free sets under this operation (i.e., with respect to fixed $C$) by defining
$$\lambda^C_{k,\ell}(G)=\max \left\{ \frac{|A|}{|G|}: (k \ast_C A)\cap (\ell \ast_C A)=\emptyset  \right\}.$$
Of course, $\lambda^{\{0\}}_{k,\ell}(G)=\lambda_{k,\ell}(G)$ recovers the ordinary definition of the maximum size of a $(k,\ell)$-sum-free set.
\\

When $G=\mathbb{Z}_n$ and $C=\{0,1\}$, we see that $\lambda^{\{0,1\}}_{k,\ell}(\mathbb{Z}_n)$ reflects the problem of finding open $(k,\ell)$-sum-free subsets of $\mathbb{T}$.  The Main Theorem for $G=\mathbb{T}$ gives $\lambda^{\{0,1\}}_{k,\ell}(\mathbb{Z}_n)\leq \frac{1}{k+\ell}$.  (This bound also follows from Kneser's Theorem for finite groups.)  Note that equality is achieved at least whenever $n$ is a multiple of $\ell^2-k^2$ (Theorem \ref{thm:lowerbound}).  This group invariant seems an interesting object of study.

\begin{question}
What can we say about $\lambda^{\{0,1\}}_{k,\ell}(\mathbb{Z}_n)$?  Which values of $n$ satisfy $\lambda^{\{0,1\}}_{k,\ell}(\mathbb{Z}_n)=\frac{1}{k+\ell}$?  For fixed $1\leq k <\ell$, which $n$ minimizes $\lambda^{\{0,1\}}_{k,\ell}(\mathbb{Z}_n)$?
\end{question}

Other compact abelian groups of the form $\mathbb{T}^d \times M$ (with $M$ finite) analogously give rise to the more general problem of computing $\lambda_{k,\ell}^C(G)$ with $G=\mathbb{Z}_{n_1}\times \cdots \times \mathbb{Z}_{n_d} \times M$ and $C=\{(\varepsilon_1, \varepsilon_2, \ldots, \varepsilon_d, 0): \varepsilon_i \in \{0,1\}\}$.  Finally, we propose that other choices of $C$ could lead to questions of future interest.

\section{Acknowledgments}
The author wishes to thank Milo Brandt and Stefan Steinerberger for reading drafts of this paper.  The author is also grateful to Sean Eberhard for suggesting the use of Pontryagin duality in the connected case.

\end{document}